\documentclass[a4paper,reqno, 12pt]{amsart}
\usepackage[left=2cm,right=2cm,top=2.5cm,bottom=2.5cm]{geometry}
\usepackage[shortlabels]{enumitem}
\usepackage{amsmath,amssymb,latexsym,esint,cite,verbatim,wasysym,mathrsfs}
\usepackage{microtype}
\usepackage{color,enumitem,graphicx}
\usepackage[colorlinks=true,urlcolor=blue, citecolor=red,linkcolor=blue,
linktocpage,pdfpagelabels,bookmarksnumbered,bookmarksopen]{hyperref}
\usepackage[hyperpageref]{backref}
\usepackage[english]{babel}

\newtheorem{theorem}{Theorem}
\newtheorem{remark}[theorem]{Remark}
\newtheorem{lemma}[theorem]{Lemma}

\newtheorem{corollary}[theorem]{Corollary}
\newtheorem{definition}[theorem]{Definition}

\newcommand{\eps}{\varepsilon}

\newcommand{\Om}{\Omega}

\newcommand*\diff{\mathop{}\!\mathrm{d}}

\numberwithin{theorem}{section} \numberwithin{equation}{section}

\title[Eigenvalue Problem]{Remarks On Eigenvalue Problems for fractional $p(\cdot)$-Laplacian}
\author[A. Bahrouni]{Anouar Bahrouni}
\address[A. Bahrouni]{Mathematics Department, University of Monastir,
Faculty of Sciences, 5019 Monastir, Tunisia}
\email{bahrounianouar@yahoo.fr}

\author[K. Ho]{Ky Ho}
\address[K. Ho]{Institute of Applied Mathematics, University of Economics Ho Chi Minh City, 59C, Nguyen Dinh Chieu Street, District 3, Ho Chi Minh City, Vietnam}
\email{\tt hnky81@gmail.com}


\begin{document}
\maketitle
\begin{abstract}
In this paper, we give some properties and remarks of the new
fractional Sobolev spaces with variable exponents. We also study the
eigenvalue problem involving the new fractional
$p(\cdot)$-Laplacian.
\end{abstract}
\section{Intoduction}
Recently, great attention has been focused on elliptic equations
involving fractional operators, both for pure mathematical research
and in view of concrete real-world applications. This type of
operator have major applications to various nonlinear problems,
including phase transitions, thin obstacle problem, stratified
materials, anomalous diffusion, crystal dislocation, soft thin
films, semipermeable membranes and flame propagation, ultra-relativistic limits of quantum mechanics, multiple scattering, minimal surfaces, material science, water waves, etc. We refer to 
 \cite{an,8,BRR} for a comprehensive introduction to the study of nonlocal problems.\\

In recent years, the study of differential equations and variational
problems involving variable exponent conditions has been an
interesting topic. Lebesgue spaces with variable exponents appeared in the
literature in 1931 in the paper by Orlicz \cite{orlicz}. Zhikov
\cite{zi} started a new direction of investigation, which created the relationship between spaces with variable exponents and variational integrals with nonstandard growth conditions. For more details on Lebesgue and Sobolev spaces with variable exponents, we refer the reader to \cite{nonlinearity,radrep}.\\

  To our best knowledge, Kaufmann et al. \cite{krv2017} firstly introduced some results on fractional Sobolev spaces with variable exponent $W^{s,q(\cdot),p(\cdot,\cdot)}(\Omega)$ and the fractional $p(\cdot)$-Laplacian. There, the authors established compact embedding theorems of these
spaces into variable exponent Lebesgue spaces. As an application,
they also proved an existence result for nonlocal problems involving
the fractional $p(\cdot)$-Laplacian. In \cite{br2018}, Bahrouni and R\v{a}dulescu obtained some further qualitative properties of the fractional
Sobolev space $W^{s,q(\cdot),p(\cdot,\cdot)}(\Omega)$ and the fractional $p(\cdot)$-Laplacian. After that, some studies on this kind of problems have been
performed by using different approaches, see \cite{an2019,zz2017,HK}.\\

Let $\Omega$ be a bounded Lipschitz domain in $\mathbb{R}^{N}$. For any real
$s>0$ and for any functions $q(x)$ and $p(x,y)$, we want to define
the fractional Sobolev space with variable exponent. We start by
fixing $s\in (0,1)$,  $q\in C(\overline{\Omega},\mathbb{R})$, and
$p\in C(\overline{\Omega}\times \overline{\Omega},\mathbb{R})$. Throughout this paper, we assume that
\begin{equation}\tag{P}\label{P}
1<p(x,y)=p(y,x)<\frac{N}{s},\ \forall\, (x,y)\in \overline{\Omega}\times \overline{\Omega}
\end{equation}
and
\begin{equation}\tag{Q}\label{Q}
1<\displaystyle q(x)<
\frac{Np(x,x)}{N-sp(x,x)}=:p_s^{\ast}(x), \ \forall\, x \in \overline{\Omega}.
\end{equation}

 We define the fractional Sobolev space with variable exponents
$W^{s,q(\cdot),p(\cdot,\cdot)}(\Omega)$ as
\begin{align*}
W^{s,q(\cdot),p(\cdot,\cdot)}(\Omega)=\left\{u\in L^{q(\cdot)}(\Omega): \  \displaystyle \int_{\Omega
\times \Omega} \frac{|u(x)-u(y)|^{p(x,y)}}{\lambda^{p(x,y)}
|x-y|^{N+sp(x,y)}}\diff x \diff y <\infty\right\}.
\end{align*}
Let
$$[u]_{s,p(\cdot,\cdot),\Om}=\inf\left\{\lambda>0: \ \displaystyle \int_{\Omega \times \Omega}
\frac{|u(x)-u(y)|^{p(x,y)}}{\lambda^{p(x,y)} |x-y|^{N+sp(x,y)}}\diff x\diff y
\leq 1\right\}$$ be the corresponding variable exponent Gagliardo
seminorm. In what follows, for brevity, we denote $W^{s,q(\cdot),p(\cdot,\cdot)}(\Omega)$ by $E$ for a general $q\in C(\overline{\Omega},\mathbb{R})$
 satisfying \eqref{Q} and by $W^{s,p(\cdot,\cdot)}(\Omega)$ when $q(x)=p(x,x)$ on $\overline{\Omega}$. Also, in some places we will write $p(x)$ instead of
 $p(x,x)$ and in this sense, $p\in C(\overline{\Omega},\mathbb{R})$. We equip $E$  with the norm
$$\|u\|_{E}=[u]_{s,p(\cdot,\cdot),\Om}+\|u\|_{L^{q(\cdot)}(\Om)}$$
(see Appendix 
 for the definitions of $L^{q(\cdot)}(\Omega)$ and  $\|\cdot\|_{L^{q(\cdot)}(\Om)}$). Then, $E$ becomes a reflexive and separable Banach space. The following embedding theorem was obtained in \cite{krv2017} for the case $q(x)>p(x,x)$ on $\overline{\Omega}$ and then was refined in \cite{zz2017,HK}.
\begin{theorem}\label{inj}
Let $\Omega \subset \mathbb{R}^{N}$ be a bounded Lipschitz domain and
let $s\in (0,1)$. Let $p\in C(\overline{\Omega}\times \overline{\Omega},\mathbb{R})$ and $q\in C(\overline{\Omega},\mathbb{R})$ satisfy $\eqref{P}$ and $\eqref{Q}$ with $q(x)\geq p(x,x)$ for all $x\in \overline{\Omega}$.
 Let $r\in C(\overline{\Omega},\mathbb{R})$ satisfy
 \begin{equation}\tag{R}\label{R}
 1<r(x)<p_s^{\ast}(x),\ \forall x \in \overline{\Omega}.
 \end{equation}
Then, there exists a constant
 $C=C(N,s,p,q,r,\Omega)$ such that
 $$\|f\|_{L^{r(\cdot)}(\Om)}\leq C\|f\|_{E},\  \forall f\in E.$$
 Thus, 
 $E$ is continuously
 embedded in $L^{r(\cdot)}(\Omega)$. Moreover, this embedding is compact.
\end{theorem}
Thanks to Theorem~\ref{inj}, under the assumptions $\eqref{P}$ and $\eqref{Q}$  with $q(x)\geq p(x,x)$ for all $x\in \overline{\Omega}$, spaces $E$ and $W^{s,p(\cdot,\cdot)}(\Omega)$ actually coincide. It is worth pointing out that for the seminorm localized on $\Om\times\Om$, there is no Poincar\'e-type inequality in general even for constant exponent case.
Because of this fact, $E$ is not suitable for studying the fractional $p(\cdot)$-Laplacian problem with Dirichlet boundary data $u=0$ in $\mathbb{R}^{N}\setminus \Omega$ via
variational methods and hence, we need to introduce another space as our solution space. In order to do this, invoking the continuity of $p$ on
$\overline{\Omega}\times \overline{\Omega}$ we extend $p$ to
$\mathbb{R}^{N}\times\mathbb{R}^{N}$ by
using Tietze extension theorem, such that $1<\inf_{(x,y)\in\mathbb{R}^{N}\times\mathbb{R}^{N}}p(x,y)\leq \sup_{(x,y)\in\mathbb{R}^{N}\times\mathbb{R}^{N}}p(x,y)<\frac{N}{s}$. We now define the following space:
$$X=\left\{u\in W^{s,p(\cdot,\cdot)}(\mathbb{R}^{N}): \ \ u=0 \ \ \mbox{on} \ \ \Omega^{c}\right\}$$
endowed with norm
$$\|u\|_{X}=[u]_{s,p(\cdot,\cdot),\mathbb{R}^N}+\|u\|_{L^{p(\cdot)}(\Om)},$$
where $ W^{s,p(\cdot,\cdot)}(\mathbb{R}^{N})$ and $[u]_{s,p(\cdot,\cdot),\mathbb{R}^N}$ are defined in the same ways as $ W^{s,p(\cdot,\cdot)}(\Omega)$ and $[u]_{s,p(\cdot,\cdot),\Omega}$ with $\Om$ replaced by $\mathbb{R}^N$. Obviously, $X$ is a closed subspace of $W^{s,p(\cdot,\cdot)}(\mathbb{R}^N)$ and hence, $(X,\|\cdot\|_X)$ is a reflexive and separable Banach space. \\

The first aim of our paper is to present some further basic results
both on the function spaces $E$ and $X$. Also, we try to improve
$X$ by giving an equivalent space (see Section~\ref{f-Sobolev spaces}).\\

Our second aim is the study of the eigenvalue problem:

\begin{equation}\label{ival}
\begin{cases}
(-\Delta)_{p(x)}^su+\alpha |u|^{p(x)-2}u+\beta |u|^{q(x)-2}u=\lambda |u|^{r(x)-2}u& \quad \text{in} \quad \Omega, \\
u=0& \quad \text{in} \quad \Bbb R^N\setminus \Omega,
\end{cases}
\end{equation}
where the operator $(-\Delta)_{p(\cdot)}^s$ is defined as
\begin{equation*}
(- \Delta)_{p(x)}^s\, u(x) := 2\ \lim_{\varepsilon \searrow 0} \int_{\mathbb{R}^N \setminus B(x,\varepsilon)} \frac{|u(x) - u(y)|^{p(x,y)-2}\, (u(x) - u(y))}{|x - y|^{N+sp(x,y)}}\, \diff y, \quad x \in \mathbb{R}^N,
\end{equation*}
where $B(x,\varepsilon):=\{z\in\mathbb{R}^N: |z-x|<\epsilon\};$ $\alpha,\beta$ are nonnegative real numbers; $\lambda>0$ is a real spectral parameter; and $r\in C(\overline{\Omega},\mathbb{R})$ satisfies \eqref{R}.

In particular, we deal with the existence, nonexistence of solutions for problem~\eqref{ival}. In the context of eigenvalue, problems involving variable exponent
represent a starting point in analyzing more complicated equations.
To our best knowledge, the first contribution in this sense is the paper by Fan et al. \cite{fan}. The authors established the existence of a sequence of eigenvalues for $p(.)$-Laplacian  $\operatorname{div}\left(|\nabla u|^{p(x)-2}\nabla u\right)$ subject to  Dirichlet boundary condition by using the Ljusternik-Schnirelmann theory. In \cite{mihai},  Mihailescu and R\v{a}dulescu studied an
eigenvalue problem with non-negative weight for the Laplace operator
on a bounded domain with smooth boundary in $\mathbb{R}^{N}, N =
3$. They showed the existence of two positive constants
$\lambda_{\ast}$ and $\lambda^{\ast}$ with $\lambda_{\ast} \leq
\lambda^{\ast}$ such that for any $\lambda \in (0,\lambda_{\ast})$
is not an eigenvalue of the problem while any $\lambda \in
(\lambda^{\ast},+\infty)$ is an eigenvalue of the problem.  Some
similar results for a class of fractional  $p(\cdot)$-Laplacian problems
involving multiple variable exponents can be found in \cite{shimi,chung}. All the aforementioned results treat only the existence of at least one solution for problem \eqref{ival} with $\alpha=1, \beta=0$ or $\alpha=0, \beta=1$. \\

This paper is organized as follows. In Section~\ref{f-Sobolev spaces}, we give some basic properties of fractional Sobolev spaces with variable exponents. In Sections~\ref{EP.Exist}, we deal with the eigenvalue problem using techniques in calculus of variations. Finally, in Appendix 
we give definitions and fundamental properties of the Lebesgue spaces with variable exponents.

\vspace{0.3cm}
\noindent\textbf{Notation}


$p^+:=\sup_{\mathbb{R}^N\times \mathbb{R}^N}\, p(x,y)$, $p^-:=\inf_{\mathbb{R}^N\times \mathbb{R}^N}\, p(x,y)$


$q^+:=\sup_{x\in\Omega}\, q(x)$, $q^-:=\inf_{x\in\Omega}\, q(x)$

$r^+:=\sup_{x\in\Omega}\, r(x)$, $r^-:=\inf_{x\in\Omega}\, r(x)$

\section{Some remarks on fractional Sobolev spaces with variables exponents}\label{f-Sobolev spaces}
 Let $\alpha,\beta\geq 0$ with $\alpha+\beta>0.$ Then, on $E$ the norm $\|\cdot\|_E$ is
equivalent to the norm
 \begin{equation}\label{equnorm}
 \|u\|_1=\inf\left\{\mu> 0:\
\rho\left(\frac{u}{\mu}\right)\leq 1\right\},
\end{equation} where $\rho: E
\rightarrow \mathbb{R}$ is defined by
$$\rho(u)=\displaystyle \int_{\Omega \times \Omega}\frac{|u(x)-u(y)|^{p(x,y)}}{p(x,y)|x-y|^{N+sp(x,y)}}\diff x\diff y+\alpha\int_{\Omega}\frac{|u|^{p(x)}}{p(x)}\diff x+\beta\int_{\Omega}\frac{|u|^{q(x)}}{q(x)}\diff x,$$

\begin{lemma}\label{fmod}
    Let $p\in C(\overline{\Omega}\times \overline{\Omega},\mathbb{R})$ and $q\in C(\overline{\Omega},\mathbb{R})$ satisfy $\eqref{P}$ and $\eqref{Q}$  with $q(x)\geq p(x,x)$ for all $x\in \overline{\Omega}$. Let $u\in E$, then the following holds:
    \begin{enumerate}
    \item[(i)]
        For $\gamma>0$, $\|u\|_1=\gamma$ if and only if $\rho(\frac{u}{\gamma})=1$;
    \item[(ii)]
        $\|u\|_1<1$ implies $\|u\|_1^{q^{+}}\leq \rho(u)\leq \|u\|_1^{p^{-}}$;
    \item[(iii)]
        $\|u\|_1>1$ implies $\|u\|_1^{p^{-}}\leq \rho(u)\leq \|u\|_1^{q^{+}}$.
    \end{enumerate}
\end{lemma}
\begin{proof}By invoking Proposition A.1 in Appendix, 
	the proof can be obtained easily from the definition of norm $\|\cdot\|_1$ and modular $\rho$ and we omit it.
\end{proof}
Next we provide some more properties on the modular $\rho$. In what follows, $E^\ast$ (resp. $X^\ast$) denotes the dual space of $E$ (resp. $X$) and $\langle \cdot,\cdot \rangle$ denote the duality pairing between $E$ and $E^\ast$ (resp. $X$ and $X^\ast$).
\begin{lemma}\label{lemma_coercivity}
     Let $p\in C(\overline{\Omega}\times \overline{\Omega},\mathbb{R})$ and $q\in C(\overline{\Omega},\mathbb{R})$ satisfy $\eqref{P}$ and $\eqref{Q}$  with $q(x)\geq p(x,x)$ for all $x\in \overline{\Omega}$. Then the following properties hold.
    \begin{enumerate}
    \item[(i)]
        The functional $\rho$ is of class $C^{1}$ and its Fr\'echet derivative $\rho': E\to E^\ast$ is given by
        \begin{align*}
        \langle \rho'(u),\varphi\rangle&=\displaystyle \int_{\Omega \times
\Omega}\frac{|u(x)-u(y)|^{p(x,y)-2}(u(x)-u(y))(\varphi(x)-\varphi(y))}{|x-y|^{N+sp(x,y)}}\diff x\diff y\\
&+\alpha\int_{\Omega} |u|^{p(x)-2}u\varphi\diff x +\beta\int_{\Omega} |u|^{q(x)-2}u\varphi\diff x, \ \  \forall
u,\varphi \in E.
        \end{align*}
    \item[(ii)]
        The function $\rho'\colon E\to E^*$ is coercive, that is, $\frac{\langle\rho'(u),u\rangle}{\|u\|_1}\to +\infty$ as $\|u\|_1\to +\infty$.
    \end{enumerate}
\end{lemma}
\begin{proof}
    (i) This is standard (see \cite{br2018}).\\

    (ii) By Lemma \ref{fmod}, for $\|u\|_1>1$, we obtain
    \begin{equation*}
    \langle\rho'(u),u\rangle\geq \rho(u)\geq \|u\|_1^{p^{-}}
    \end{equation*}
    and hence, the conclusion follows.
    \end{proof}
\begin{remark}\rm
The above results still hold true if we replace $E$ by $X$ with $\int_{\Omega \times
    \Omega}$ replaced by $\int_{\mathbb{R}^N \times
    \mathbb{R}^N}$.
\end{remark}

As we mentioned in the introduction, when $\alpha=\beta=0$, we need
to use the space $X$ instead $E$ to study problem \eqref{ival} via
variational methods. For this purpose, in the rest of
this section we will provide further properties for the space $X$. As we discussed in the introduction, for $p\in C(\overline{\Omega}\times \overline{\Omega},\mathbb{R})$ satisfying \eqref{P} we can extend $p$ the whole space $\mathbb{R}^{N}\times \mathbb{R}^{N}$ to have $p\in
C(\mathbb{R}^{N}\times \mathbb{R}^{N},\mathbb{R})$ satisfying \eqref{P} with $\Omega=\mathbb{R}^{N}$. We introduce a new norm on $X$ as follows:
$$
\aligned \|u\|_{0}
&=\inf\left\{\lambda>0 : \int_{\mathbb{R}^{N}}\int_{\mathbb{R}^{N}}
\frac{|u(x)-u(y)|^{p(x, y)}}{\lambda^{p(x, y)}|x-y|^{N+s p(x, y)}}
\diff x \diff y \leq 1\right\}.
\endaligned
$$
\begin{lemma}\label{lemma2-6}
The functional $M:\, X\to \mathbb{R}$ defined by
$$
M(u)=\int_{\mathbb{R}^{N}}\int_{\mathbb{R}^{N}}
 \frac{|u(x)-u(y)|^{p(x, y)}}{|x-y|^{N+s p(x, y)}} \diff x\diff y
 $$
has the following properties:
\begin{enumerate}
     \item[(i)] for $\alpha> 0$, $\|u\|_0=(>,<)\,\alpha $ if and only if $M(\frac{u}{\alpha})=(>,<)\, 1$;
  \item [(ii)]if $\| u\|_0>1$, then
 $\| u\|_0^{p^{-}} \leq
M(u) \leq \| u\|_0^{p^{+}}$.
 \item[(iii)]
 if $\| u\|_0<1$, then
 $\| u\|_0^{p^{+}} \leq
M(u) \leq \| u\|_0^{p^{-}}$.
\end{enumerate}
\end{lemma}
\begin{proof}
It is a direct consequence of Proposition A.1.
\end{proof}
Now, we prove the following compact embedding type result by employing some ideas in  \cite{an2019}.
\begin{theorem}\label{th2-8}
Let $\Omega \subset \mathbb{R}^{N}$ be a bounded Lipschitz domain and
let $s\in (0,1)$. Let $ p\in
C(\mathbb{R}^{N}\times \mathbb{R}^{N},\mathbb{R})$ satisfy $\eqref{P}$
(with $\Omega=\mathbb{R}^{N}$). Then, for any $r\in
C(\overline{\Omega},\mathbb{R})$ satisfying \eqref{R}, there exists a constant $C>0$ such that
 \begin{equation}\label{eq2-5}
 \|u\|_{L^{r(\cdot)}(\Omega)}\leq C\|u\|_0,\quad \forall u\in X.
\end{equation}
Moreover, the embedding $X\hookrightarrow L^{r(\cdot)}(\Omega)$ is compact.
\end{theorem}
\begin{proof}
First, we claim that there exists a constant $C_{0}>0$ such that
\begin{equation}\label{eq2-6}
C_{0}\|u\|_{L^{p(\cdot)}(\Omega)}\leq \|u\|_0,\quad \forall
u\in X,
\end{equation}
To this end, let
$\mathcal{A}=\{u\in X:\|u\|_{L^{p(\cdot)}(\Omega)}=1\}$. Take
a sequence $\{u_{n}\}\subset \mathcal{A}$ such that
$\lim\limits_{n\to \infty}\|u_{n}\|_0=\inf\limits_{u\in
\mathcal{A}}\|u\|_0$. So, $\{u_{n}\}$ is bounded in
$L^{p(\cdot)}(\Omega)$ and $X_0$. Hence, $\{u_{n}\}$ is bounded in
$W^{s,p(\cdot,\cdot)}(\Omega)$. Up to a subsequence, there exist a
subsequence of $\{u_{n}\}$, still denote by $\{u_{n}\}$, and
$u_{0}\in W^{s,p(\cdot,\cdot)}(\Omega)$ such that
$u_{n}\rightharpoonup u_{0}$ 
 in $W^{s,p(\cdot,\cdot)}(\Omega)$. By Theorem \ref{inj}, we get that
$u_{n}\to u_{0}$ 
in $L^{p(\cdot)}(\Omega)$ and
$\|u_{0}\|_{L^{p(\cdot)}(\Omega)}=1$. Now, we extend $u_{0}$ to
$\mathbb{R}^{N}$ by setting $u_{0}=0$ in
$\mathbb{R}^{N}\backslash\Omega$. This implies $u_{n}(x)\to u_{0}(x)$
a.e. in $\mathbb{R}^{N}$ as $n\to \infty$. Hence, by Fatou's Lemma,
we have
\begin{equation*}
\int_{\mathbb{R}^{N}}\int_{\mathbb{R}^{N}}
\frac{|u_{0}(x)-u_{0}(y)|^{p(x, y)}}{|x-y|^{N+s p(x, y)}} \diff x\diff y
\leq \liminf\limits_{n\to \infty}
\int_{\mathbb{R}^{N}}\int_{\mathbb{R}^{N}}
\frac{|u_{n}(x)-u_{n}(y)|^{p(x, y)}}{|x-y|^{N+s p(x, y)}} \diff x\diff y,
\end{equation*}
which joining with $\|u_{0}\|_{L^{p(\cdot)}(\Omega)}=1$ implies that $u_0\in \mathcal{A}$. Set $\lambda_0:=\inf\limits_{u\in
    \mathcal{A}}\|u\|_0$ and $\lambda_n:=\|u_{n}\|_0$ ($n=1,2\cdots$). From the fact that $u_n\in \mathcal{A}$
     we have that $\lambda_n>0$ for $n=0,1,2,\cdots$ and hence, by Lemma~\ref{lemma2-6} and by Fatou's Lemma again, we have
\begin{align*}
\int_{\mathbb{R}^{N}}\int_{\mathbb{R}^{N}}&
\frac{|u_{0}(x)-u_{0}(y)|^{p(x, y)}}{\lambda_0^{p(x,y)}|x-y|^{N+s p(x, y)}} \diff x\diff y\\
&
\leq \liminf\limits_{n\to \infty}
\int_{\mathbb{R}^{N}}\int_{\mathbb{R}^{N}}
\frac{|u_{n}(x)-u_{n}(y)|^{p(x, y)}}{\lambda_n^{p(x,y)}|x-y|^{N+s p(x, y)}} \diff x\diff y=1.
\end{align*}
This and Lemma~\ref{lemma2-6} yield
$$\|u_{0}\|_0\leq \lambda_0=\inf\limits_{u\in
    \mathcal{A}}\|u\|_0.$$
Therefore, we obtain
$0<\|u_{0}\|_0=\inf\limits_{u\in \mathcal{A}}\|u\|_0:=C_{0}$
and this proves our claim. From \eqref{eq2-6}, it follows that
\begin{equation}\label{eq2-7}
\aligned \|u\|_{W^{s,p(\cdot,\cdot)}(\Omega)}
&= \|u\|_{L^{p(\cdot)}(\Omega)}+[u]_{W^{s,p(\cdot,\cdot)}(\Omega)}\\
&\leq \|u\|_{L^{p(\cdot)}(\Omega)}+\|u\|_0\leq
(1+\frac{1}{C_0})\|u\|_0,
\endaligned
\end{equation}
which implies that $X$ is continuously embedded in
$W^{s,p(\cdot,\cdot)}(\Omega)$. From \eqref{eq2-7} and Theorem \ref{inj}, there exists
a constant $C>0$ such that
$$
\|u\|_{L^{r(\cdot)}(\Omega)}\leq C\|u\|_0.
$$
Thus, \eqref{eq2-5} has been proved. Finally, combining the fact that $X\hookrightarrow W^{s,p(\cdot,\cdot)}(\Omega)$ and $W^{s,p(\cdot,\cdot)}(\Omega)\hookrightarrow \hookrightarrow L^{r(\cdot)}(\Omega)$ (applying Theorem~\ref{inj} again) we obtain $X \hookrightarrow \hookrightarrow L^{r(\cdot)}(\Omega)$. The proof is complete.
\end{proof}
\section{Eigenvalue problem}\label{EP.Exist}
Motivated by \cite{fan, Mih-Rad.MM2008}, in this section we are concerned with the following nonhomogeneous problem
\begin{equation}\label{val}
\begin{cases}
(-\Delta)_{p(x)}^su+\alpha |u|^{p(x)-2}u+\beta |u|^{q(x)-2}u=\lambda |u|^{r(x)-2}u& \quad \text{in} \quad \Omega, \\
u=0& \quad \text{in} \quad \Bbb R^N\setminus \Omega,
\end{cases}
\end{equation}
where $ p\in
C(\mathbb{R}^{N}\times \mathbb{R}^{N},\mathbb{R})$ satisfy $\eqref{P}$ (with $\Omega=\mathbb{R}^{N}$), $q,r\in
C(\overline{\Omega},\mathbb{R})$ satisfy \eqref{Q} and \eqref{R}; $\alpha,\beta$ are nonnegative real numbers; and $\lambda$ is a
real spectral parameter.

Note that when $\beta =0,$ we will regard $q(x)=p(x)$ on $\overline{\Omega}$ in all our statements appearing $q$.
\begin{definition}\rm
A pair $(u,\lambda)\in X\times\mathbb{R}$ is called a
solution of problem \eqref{val} if
\begin{align*}
\displaystyle \int_{\mathbb{R}^{N}\times
\mathbb{R}^{N}}&\frac{|u(x)-u(y)|^{p(x,y)-2}(u(x)-u(y))(v(x)-v(y))}{|x-y|^{N+sp(x,y)}}\diff
x\diff y +\alpha \displaystyle \int_{\Omega} |u|^{p(x)-2}uv\diff
x\\&+\beta \displaystyle \int_{\Omega} |u|^{q(x)-2}uv\diff x=\lambda
\displaystyle \int_{\Omega} |u|^{r(x)-2}uv\diff x,\quad \forall
v\in X.
\end{align*}
If $(u,\lambda)$ is a solution of problem \eqref{val} and
$u\in X\setminus \{0\}$, as usual, we call $\lambda$
and $u$ an \textit{eigenvalue} and an \textit{eigenfunction}
corresponding to $\lambda$ for problem~ \eqref{val}, respectively. A
solution $(u,\lambda)$ of \eqref{val} with $u\ne 0$ is also called an
\textit{eigenpair} of problem~\eqref{val}. 
\end{definition}
 We divide this section into three subsections. In the first part by employing the Ljusternik-Schnirelmann theory, we construct
  a sequence of eigenvalues of problem~\eqref{val}. 
  We also discuss about the positivity of the infimum of the set of eigenvalues of problem~\eqref{val} in this subsection. In the last two parts,  we deal
   with the existence and the nonexistence of eigenvalues of problem~ \eqref{val} under some additional assumptions.

 In order to investigate the eigenvalues of \eqref{val}, we consider the energy functional associated with problem~\eqref{val}. In particular we consider the
 functionals
  $I,I_0, J, J_0,\Phi_\lambda:\, X \rightarrow \mathbb{R}$ given by
 $$I(u)=\displaystyle
 \int_{\mathbb{R}^{N}\times
    \mathbb{R}^{N}}\frac{|u(x)-u(y)|^{p(x,y)}}{p(x,y)|x-y|^{N+sp(x,y)}}\diff x\diff y
 + \displaystyle \alpha\int_{\Omega} \frac{|u|^{p(x)}}{p(x)}\diff x+
 \displaystyle \beta\int_{\Omega} \frac{|u|^{q(x)}}{q(x)}\diff x,$$
 $$I_0(u)=\displaystyle
 \int_{\mathbb{R}^{N}\times
    \mathbb{R}^{N}}\frac{|u(x)-u(y)|^{p(x,y)}}{|x-y|^{N+sp(x,y)}}\diff x\diff y
 + \displaystyle \alpha\int_{\Omega} |u|^{p(x)}\diff x+ \displaystyle
 \beta\int_{\Omega} |u|^{q(x)}\diff x,$$
 $$J(u)=\displaystyle \int_{\Omega}
 \frac{|u|^{r(x)}}{r(x)}\diff x,\quad J_0(u)=\displaystyle \int_{\Omega}
 |u|^{r(x)}\diff x,$$
 and
 \begin{equation}\label{Phi_lambda}
 \Phi_\lambda(u)=I(u)-\lambda J(u).
 \end{equation}
Invoking Theorem~\ref{th2-8} and a standard argument, we can show that $I,I_0, J, J_0,\Phi_\lambda\in C^1(X,\mathbb{R})$ and a critical point of $\Phi_\lambda$ is a solution to problem~\eqref{val}. In the sequel, we will make use of the following values:
\begin{equation}\label{def.gamma0,gamma1}
\gamma_1=\displaystyle \inf_{u\in X\setminus\{0\}} \frac{I(u)}{J(u)}\ \ \text{and}\ \  \gamma_0=\displaystyle \inf_{u\in X\setminus\{0\}} \frac{I_0(u)}{J_0(u)}.
\end{equation}
 Clearly,
 \begin{equation}\label{relation.gamma0-gamma1}
\frac{\min\{p^-,q^-\}}{r^+}\gamma_1 \leq \gamma_0\leq \frac{\max\{p^+,q^+\}}{r^-}\gamma_1.
 \end{equation}
 In what follows, unless otherwise stated on $X$ we will make use the norm $\|\cdot\|_0$. By Theorem~\ref{th2-8},
 this norm is equivalent to $\|\cdot\|_X$ or $\|\cdot\|_1$ given by \eqref{equnorm} when $\alpha+\beta>0.$ 
\subsection{A sequence of eigenvalues } In this subsection we construct a sequence of eigenvalues for \eqref{val} via the Ljusternik-Schnirelmann theory. Denote
\begin{equation*}\label{def.Lambda}
\Lambda:=\{\lambda:\ \lambda\ \text{is an eigenvalue of } \eqref{val}\}.
\end{equation*}
For $t>0$, define
\begin{equation*}
N_t:=\{u\in X:\ I(u)=t\}.
\end{equation*}
Clearly, for each $u\in X\setminus\{0\}$, there exists a
unique $s_t=s_t(u)\in (0,\infty)$ such that $s_tu\in N_t.$
 Moreover, we have
 \begin{equation}\label{s_t}
 s_t\to 0\ \ \text{as} \ \ t\to 0^+\ \ \text{and}\ \ s_t\to+\infty\ \ \text{as} \ \ t\to+\infty.
 \end{equation}
 For each $n\in\mathbb{N},$ define
\begin{equation*}
\mathcal{K}_n:=\{K\subset X\setminus\{0\}:\ K\ \text{is compact }, -K=K,\ \text{and } \gamma (K)\geq n\},
\end{equation*}
where $\gamma (K)$ denote the Krasnoselski genus of $K$, and
\begin{equation*}
c_n(t):=\underset{K\subset
N_t}{\underset{K\in\mathcal{K}_n}{\sup}}\, \inf_{u\in K}\, J(u).
\end{equation*}
Clearly, $c_n(t)$ is well defined for all $n\in\mathbb{N}$. Moreover, we have
\begin{equation}\label{seq.cn}
c_1(t)\geq c_2(t)\geq \cdots\geq c_n(t)\geq c_{n+1}(t)\geq\cdots>0.
\end{equation}
Also, we have the following formular
\begin{equation}\label{c1}
c_1(t)=\sup_{u\in N_t}\, J(u).
\end{equation}
By the Lagrange multiplier rule, $u$ is a critical point of $J$ restricted to $N_t$ if and only if  $(u,\lambda)$ with
\begin{equation}\label{form.lambda(u)}
\lambda=\lambda(u):=\frac{I_0(u)}{J_0(u)}
\end{equation}
is a solution of \eqref{val} (see \cite[Sections 43.9 and 44.5]{Zeidler}). 
The next theorem is deduced from the Ljusternik-Schnirelmann theory (see \cite[Theorem 44.A]{Zeidler}).
\begin{theorem}\label{L-S}
    For each $t>0$, the following assertions hold:
    \begin{itemize}
        \item [(i)] for each $n\in\mathbb{N},$ $c_n(t)$ is a critical value of $J$ restricted on $N_t$;
        \item [(ii)] $c_n(t)\to 0^+$ as $n\to\infty.$
    \end{itemize}
\end{theorem}
Let $u_n\in N_t$ such that $c_n(t)=J(u_n)$, then by \eqref{form.lambda(u)}, $(u_n,\lambda_n)$ is an eigenpair of \eqref{val} with
\begin{equation*}
\lambda_n=\frac{I_0(u_n)}{J_0(u_n)}\geq \frac{\min\{p^-,q^-\}}{r^+}\frac{I(u_n)}{J(u_n)}=\frac{\min\{p^-,q^-\}}{r^+}\frac{t}{c_n(t)}.
\end{equation*}
Hence, the next corollary is a direct consequence of Theorem~\ref{L-S}.
\begin{corollary}
For each $t>0$, problem~\eqref{val} admits a sequence of eigenpairs $\{(u_n,\lambda_n)\}$ with $u_n\in N_t$ and $\lambda_n\to +\infty$ as $n\to \infty$.
\end{corollary}  
To have more information about the set of eigenpairs associated with $c_n(t)$ resticted to $N_t$, define for $t>0$ and $n\in\mathbb{N},$
\begin{equation*}
K_n(t):=\{u\in N_t:\ u\ \text{ is a critical point of}\ J \text{ restrited to }\ N_t\ \text{ and } J (u)=c_n(t)\}
\end{equation*}
and \begin{equation*}
\Lambda_n(t):=\{\lambda(u):\ u\in K_n(t)\}.
\end{equation*}
By \eqref{form.lambda(u)} again, we have
\begin{equation}\label{lambda(u)}
\frac{\min\{p^-,q^-\}}{r^+}\frac{t}{c_n(t)}\leq\lambda(u)=\frac{I_0(u)}{I_0(u)}\leq \frac{\max\{p^+,q^+\}}{r^-}\frac{t}{c_n(t)},\quad \forall u\in K_n(t).
\end{equation}
In the following, for brevity, an inequality $\Lambda_n(t)\leq (\geq)\ C$ means that $\lambda \leq (\geq)\ C$ for every $\lambda\in \Lambda_n(t)$
 and a limit $\Lambda_n(t)\to a$ as $n\to\infty$ means the limit occurs unfiromly with respect to $\lambda\in \Lambda_n(t)$. By the definitions of
 $\Lambda_n(t)$ and \eqref{lambda(u)}, we easily obtain the following estimates: for each $t>0$ and $n\in \mathbb{N}$,
\begin{equation}\label{Lambda_n-mu_n}
\frac{\min\{p^-,q^-\}}{r^+}\frac{t}{c_n(t)}\leq \Lambda_n(t)\leq \frac{\max\{p^+,q^+\}}{r^-}\frac{t}{c_n(t)}.
\end{equation}
From \eqref{Lambda_n-mu_n} and Theorem~\ref{L-S}, we have the following.
\begin{theorem} \label{propertiesLambda_n}
    For each $t>0$ and for each $n\in\mathbb{N},$ the sets $K_n(t)$ and $\Lambda_n(t)$ are nonempty, $\Lambda_n(t)\subset\Lambda$, and for any
    $u\in K_n(t),$ $(u,\lambda(u))$ is a solution of \eqref{val}. Moreover, for each $t>0,$ $\Lambda_n(t)\to +\infty$ as $n\to\infty.$
\end{theorem}

\vspace{0.3cm}

\noindent\textbf{The infimum of eigenvalues}

Denote
\begin{equation}\label{lambda_*}
\lambda_*:=\inf\,\Lambda.
\end{equation}
By Theorem~\ref{propertiesLambda_n}, $\lambda_*$ is well defined and it is clear that $\lambda_*\in [0,\infty)$. It is worth pointing out that when $p$ $q$, and $r$ are constant functions and $p=q=r$, we have
that $\lambda_*=\gamma_0>0$ and is the first eigenvalue of
\eqref{val}. In the variable exponent case, it is not true in
general.    First, we have the relation of positivity of
$\gamma_0,\gamma_1$ and $\lambda_*$ as follows.
\begin{lemma}\label{relations.lambda1}
   It holds that
    $$\gamma_1>0 \iff \gamma_0>0 \iff\lambda_*>0.$$
\end{lemma}
\begin{proof}
   By \eqref{relation.gamma0-gamma1}, it suffices to prove that
   \begin{equation}\label{gamma_0-lambda_*}
  \gamma_0>0 \iff\lambda_*>0.
   \end{equation}
     It is clear that  $\gamma_0\leq \lambda_*$ hence; $(\Rightarrow)$ is obvious. Now suppose $\gamma_0=0$.
    Then for any given $\epsilon>0,$ there exists $u_\epsilon\in X\setminus\{0\}$ such that
    \begin{equation}\label{eps1}
    \frac{I_0(u_\eps)}{J_0(u_\eps)}<\eps.
    \end{equation}
    Let $I(u_\eps)=t.$ Then, from \eqref{seq.cn}, \eqref{c1} and \eqref{eps1} we obtain
    \begin{equation}\label{eps2}
    \frac{t}{c_1(t)}\leq \frac{I(u_\eps)}{J(u_\eps)}\leq \frac{\frac{1}{\min\{p^-,q^-\}}I_0(u_\eps)}{\frac{1}{r^+}J_0(u_\eps)}<\frac{r^+}{\min\{p^-,q^-\}}\eps.
    \end{equation}
    Let $u_t\in N_t$ such that $J(u_t)=c_1(t).$ Then, we deduce from \eqref{eps2} that
        \begin{equation*}
    \lambda(u_t)=\frac{I_0(u_t)}{J_0(u_t)}\leq \frac{\max\{p^+,q^+\}I(u_t)}{r^-J(u_t)}=\frac{\max\{p^+,q^+\}}{r^-}\frac{t}{c_1(t)}<\frac{r^+\max\{p^+,q^+\}}{r^-\min\{p^-,q^-\}}\eps.
    \end{equation*}
    Combining this with \eqref{lambda_*} gives
     \begin{equation*}
    0\leq \lambda_*<\frac{r^+\max\{p^+,q^+\}}{r^-\min\{p^-,q^-\}}\eps.
    \end{equation*}Since $\eps>0$ was chosen arbitrarily, we arrive at $\lambda_*=0.$ This infers that $(\Leftarrow)$ also holds. That is, \eqref{gamma_0-lambda_*} holds and the proof is complete.
\end{proof}
In the next two lemmas, we provide sufficient conditions to get $\lambda_*=0.$ We will make use of the following conditions. In these conditions,
 by $h^+(V)$ (resp. $h^-(V)$) we mean the supremum (resp. infimum) of the function $h$ over the set $V$.
\begin{itemize}
    \item [(A1)] There exist an open subset $U$ of $\Omega$ such that
    $$r^+(U)<\min\left\{p^-(U\times \mathbb{R}^N),q^-(U)\right\}.$$
    \item [(A2)] There exist an open subset $\widetilde{U}$ of $\Omega$ such that
    $$r^-(\widetilde{U})>\max\left\{p^+(\widetilde{U}\times \mathbb{R}^N),q^+(\widetilde{U})\right\}.$$
    \end{itemize}
For each $t>0$, define
\begin{equation*}\label{mu_1(t)}
\mu_1(t):=\frac{t}{c_1(t)}
\end{equation*}
and
\begin{equation}\label{lambda_*(t)}
\lambda^*(t):=\inf\{\lambda(u):\, u\ \text{ is a critical point of}\ J \text{ restrited to }\ N_t\}.
\end{equation}
\begin{lemma}\label{Sufficitent.Condition1}
    Let $(\textup{A}1)$ hold. Then, $\mu_1(t)\to 0^+$ and $\lambda^*(t)\to 0^+$ as $t\to 0^+$. Consequently, $\lambda_*=0.$
\end{lemma}
\begin{proof}
    Let $B$ be a ball in $\mathbb{R}^N$ such that $\overline{B}\subset U$. Let $\varphi\in C_c^\infty(\Om)$ be such that $\varphi\equiv 1$
     on $B$ and $\varphi\equiv 0$ on $\mathbb{R}^N\setminus U.$ By \eqref{s_t} and the strictly increasing monotonicity of $t\mapsto I(tu)$ on $(0,+\infty)$ for each $t>0$ small enough, there exists a unique $s_t\in (0,1)$ such that
     $s_t\varphi\in N_t.$ Set $\delta:= \min\left\{p^-(U\times \mathbb{R}^N),q^-(U)\right\}-r^+(U)>0.$ 
     Let $t\in (0,+\infty)$ be arbitrary and fixed. We have
    \begin{equation}\label{mu_1}
   \mu_1(t)\leq \frac{\int_{\mathbb{R}^{N} \times
            \mathbb{R}^{N}}\frac{s_t^{p(x,y)}}{p(x,y)}\frac{|\varphi(x)-\varphi(y)|^{p(x,y)}}{|x-y|^{N+sp(x,y)}}\diff x\diff y+\alpha\int_{\Omega}\frac{s_t^{p(x)}}{p(x)}|\varphi|^{(x)}\diff x+\beta\int_{\Omega}\frac{s_t^{q(x)}}{q(x)}|\varphi|^{q(x)}\diff x}
            {\int_{\Omega} \frac{s_t^{r(x)}}{r(x)}||\varphi(x)|^{r(x)}\diff
            x}.
    \end{equation}
    We estimate each integral in the right-hand side of \eqref{mu_1} as follows. We have
    \begin{align*}
    \int_{\mathbb{R}^{N} \times
        \mathbb{R}^{N}}\frac{s_t^{p(x,y)}}{p(x,y)}\frac{|\varphi(x)-\varphi(y)|^{p(x,y)}}{|x-y|^{N+sp(x,y)}}\diff x\diff y&
        \leq 2\int_{\mathbb{R}^N}\int_U s_t^{p(x,y)} \frac{|\varphi(x)-\varphi(y)|^{p(x,y)}}{p(x,y)|x-y|^{N+sp(x,y)}}\diff x\diff y\notag\\
    & \leq 2s_t^{r^+(U)+\delta}\int_{\mathbb{R}^N}\int_U \frac{|\varphi(x)-\varphi(y)|^{p(x,y)}}{p(x,y)|x-y|^{N+sp(x,y)}}\diff x\diff
    y.
    \end{align*}
    For the second and the third integrals, we estimate
    \begin{equation*}
    \int_{\Omega}\frac{s_t^{p(x)}}{p(x)}|\varphi|^{p(x)}\diff x= \int_{U}\frac{s_t^{p^-}|\varphi|^{p(x)}}{p(x)}\diff x\leq s_t^{r^+(U)+\delta}\int_{U}\frac{|\varphi|^{p(x)}}{p(x)}\diff
    x,
    \end{equation*}
    \begin{equation*}
   \int_{\Omega}\frac{s_t^{q(x)}}{q(x)}|\varphi|^{q(x)}\diff x\leq \int_{U}\frac{s_t^{q^-}|\varphi|^{q(x)}}{q(x)}\diff x\leq s_t^{r^+(U)+\delta}\int_{U}\frac{|\varphi|^{q(x)}}{q(x)}\diff
   x.
    \end{equation*}
    Finally, we have
    \begin{equation*}
    \int_{\Omega} \frac{s_t^{r(x)}}{r(x)}|\varphi(x)|^{r(x)}\diff x\geq
    s_t^{r^+(U)}\int_{\Omega} \frac{|\varphi(x)|^{r(x)}}{r(x)}\diff x.
    \end{equation*}
    Utilizing the last four estimates, we deduce from \eqref{mu_1} that
    \begin{equation*}
    \mu_1(t)\leq \frac{2I(\varphi)}{J(\varphi)} s_t^\delta.
    \end{equation*}
    On the other hand, by \eqref{lambda(u)}
    and \eqref{lambda_*(t)} we have
    \begin{equation*}
    0\leq \lambda^*(t)\leq \frac{\max\{p^+,q^+\}}{r^-}\mu_1(t).
    \end{equation*}
   Combining the last two estimates and \eqref{s_t}, we conclude the lemma. The proof is complete.
    \end{proof}

    \begin{lemma}\label{Sufficitent.Condition2}
        Let $(\textup{A}2)$ hold. Then, $\mu_1(t)\to 0^+$ and $\lambda^*(t)\to 0^+$ as $t\to +\infty$. Consequently, $\lambda_*=0.$
    \end{lemma}
    The proof of Lemma~\ref{Sufficitent.Condition2} is similar to that of Lemma~\ref{Sufficitent.Condition1} for which we put $\delta:=r^-(\widetilde{U})-\max\left\{p^+(\widetilde{U}\times \mathbb{R}^N),q^+(\widetilde{U})\right\}$ and take $t\in (0,+\infty)$ so large that $s_t\in (1,+\infty).$ We leave the details to the reader.
    
    Set
    \begin{equation}\label{mu*}
   \mu_\ast:=\inf \{\mu_1(t):\ t\in (0,\infty)\}\ \ \text{and}\ \ \mu^\ast:=\sup \{\mu_1(t):\ t\in (0,\infty)\}.
    \end{equation}
    By Lemmas~\ref{Sufficitent.Condition1} and \ref{Sufficitent.Condition2}, if either $(\textup{A}1)$ or $(\textup{A}2)$ holds,
    then $\mu_\ast=0.$ Clearly, we always have $\mu^\ast>0.$ Moreover, if $(\textup{A}1)$ and $(\textup{A}2)$ hold, then $\mu^\ast<\infty$.
 
\subsection{Existence results with the growth of type I}

In this subsection, we provide a precise range of small eigenvalues for problem \eqref{val}.  Using the positive constant $\mu^*$ given by \eqref{mu*}, we have the following existence result.
\begin{theorem}\label{RangeEigenvalues}
    Let $(\textup{A}1)$ hold and define $\Phi_\lambda$ as in \eqref{Phi_lambda}. Then, for any given $\lambda\in (0,\mu^\ast)$, $\Phi_\lambda$ has a nonnegative local minimizer $u_\lambda$ such that $\Phi_\lambda(u_\lambda)<0$. Consequently, for any given $\lambda\in (0,\mu^\ast)$, problem~\eqref{val} has a nontrivial nonnegative solution $u_\lambda$ with $\Phi_\lambda(u_\lambda)<0$.
\end{theorem}
\begin{remark}\rm It is worth pointing out in existing works close to our work (e.g. \cite{chung, shimi,MR.2007}) the author assumed $q(x)\geq p(x,x)$ on $\overline{\Omega}$ and used a sublinear type
growth $r^+<p^-$ or a mixed condition $\displaystyle r^-<p^-<r^+.$
It is easy to see that with the additional assumption  $q(x)\geq p(x,x)$ on $\overline{\Omega}$, the condition $r^-<p^-$ implies the condition
$(\textup{A}1)$. That is, we are dealing with a weaker condition for this type of growth.
\end{remark}
The proof of Theorem~\ref{RangeEigenvalues} is similar to that of \cite[Theorem 3.3]{Fan.JMAA.2009} and we only sketch the proof for sake of completeness.
\begin{proof}[Proof of Theorem~\ref{RangeEigenvalues}] Let $\lambda\in (0,\mu^\ast)$.
By Lemma~\ref{Sufficitent.Condition1}, we have $\mu_*=0.$ Since $\mu_1(t)$ is coninuous with respect to $t$ on $(0,+\infty)$ (c.f. \cite[Proposition 2.3]{Fan.JMAA.2009}), $\mu_1((0,+\infty))$ is connected. Thus, we find $t_\lambda>0$ such that
$$\lambda\leq\mu_1(t_\lambda)=\frac{t_\lambda}{c_1(t_\lambda)}.$$
Set $D=\{u\in X:\, I(u)\leq t_\lambda\}$. Then, $D$ is closed, bounded and convex in $X$ and $\partial D=I^{-1}(t_\lambda)=N_{t_\lambda}.$
Invoking \eqref{c1} with $t=t_\lambda$ it follows that for any $u\in\partial D,$
$$\Phi_\lambda(u)\geq t_\lambda-\lambda  c_1(t_\lambda)=t_{\lambda}\left(1-\frac{\lambda c_1(t_\lambda)}{t_\lambda}\right)\ge 0.$$
Since
$\Phi_\lambda:\, D\to\mathbb{R}$ is weakly lower semicontinuous on
$D$ and $D$ is weakly compact, $\Phi_\lambda$ achieves a global
minimum on $D$ at some $w_\lambda\in D$ i.e.,
$$ \Phi_\lambda(w_\lambda)=\inf_{u\in D}\Phi_\lambda(u).$$
We claim that $\Phi_\lambda(w_\lambda)<0$. To this end, invoking
Lemma~\ref{Sufficitent.Condition1} again  we find $t_0\in
(0,t_\lambda)$ such that $\mu_1(t_0)<\lambda$.
 Let $v\in N_{t_0}$ such that $J(v)=c_1(t_0)$. This yields
$$\Phi_\lambda(v)=I(v)-\lambda J(v)=t_0-\lambda c_1(t_0)<0$$
which shows that $\Phi_\lambda (w_\lambda)=\inf_{u\in D}\Phi_\lambda(u)<0$. By letting $u_\lambda=|w_\lambda|$, we deduce that $u_\lambda$ is a local minimizer of $D$ and hence,
 $u_\lambda$ is a nontrivial nonnegative solution of problem~\eqref{val}. The proof is complete.
\end{proof}
\subsection{Existence/Nonexistence results with the growth of type II}
In this part we study the nonexistence of eigenvalue for problem
\eqref{val} with $\alpha>0$ and $\beta>0$ assuming that the functions $p,q$ and $r$ satisfy the
condition
\begin{equation}\tag{G}\label{G}
p^+<r^-\leq r^+<q^-\leq q^+<\frac{Np^-}{N-sp^-}.
\end{equation}
Our main result in this subsection is giving by the following theorem.
\begin{theorem}\label{mainnon}
    Assume that conditions \eqref{P} (with $\Omega$ replaced by $\mathbb{R}^N$) and \eqref{G} are fulfilled and let  $\gamma_0,\gamma_1$ be defined in \eqref{def.gamma0,gamma1}. Then, $\gamma_0,\gamma_1\in (0,\infty)$ and any $\lambda \in (\gamma_1,\infty)$ is
    an eigenvalue of problem \eqref{val} which admits a nonnegative eigenfunction and any $\lambda \in (0,\gamma_0)$ is
    not an eigenvalue of problem \eqref{val}.
\end{theorem}
In the rest of this subsection, on $X$ we will make use of the equivalent norm
$$\|u\|:=\inf\left\{\lambda> 0:\
\int_{\mathbb{R}^N \times \mathbb{R}^N}\frac{|u(x)-u(y)|^{p(x,y)}}{\lambda^{p(x,y)}|x-y|^{N+sp(x,y)}}\diff x\diff y+\int_{\Omega}\left|\frac{u}{\lambda}\right|^{p(x)}\diff x\leq 1\right\}.$$
We now prove Theorem~\ref{mainnon} by adapting ideas used in \cite{Mih-Rad.MM2008}. In the rest of this section, we always that assumptions of Theorem \ref{mainnon} are fulfilled and for simplicity and clarity of our arguments, we just take $\alpha=\beta=1$.
\begin{lemma}
 It holds that
$$\gamma_1,\gamma_0>0.$$
\end{lemma}
\begin{proof}
By condition \eqref{G} we deduce that
$$\int_{\Omega}|u|^{q(x)}\diff x+\int_{\Omega}|u|^{p(x)}\diff x\geq\int_{\Omega}|u|^{r(x)}\diff x.$$
Thus
 $$
\displaystyle\int_{\mathbb{R}^{N}\times
\mathbb{R}^{N}}\frac{|u(x)-u(y)|^{p(x,y)}}{|x-y|^{N+sp(x,y)}}\diff x\diff y +
\displaystyle \int_{\Omega} |u|^{p(x)}\diff x+\displaystyle
\int_{\Omega} |u|^{q(x)}\diff x\geq\int_{\Omega}|u|^{r(x)}\diff x. $$ This
implies that $\gamma_0>0$. This and \eqref{relation.gamma0-gamma1} imply $\gamma_1>0$ and the proof is complete.
\end{proof}
\begin{lemma} It holds that
    \begin{equation}\label{lim.infty}
    \displaystyle \lim_{\|u\|\to+\infty} \frac{I(u)}{J(u)}=+\infty
    \end{equation}
    and
    \begin{equation}\label{lim.0}
\displaystyle \lim_{\|u\|\to 0^+} \frac{I(u)}{J(u)}=+\infty.
    \end{equation}
\end{lemma}
\begin{proof}
    We first note that $r^+<q^-$ and the embedding $X\hookrightarrow L^{r(\cdot)}(\Omega)$ imply that there is $C_{r}>1$ such that
    \begin{equation}\label{4.2.Emb}
   \|u\|_{L^{r(\cdot)}(\Omega)}\leq C_{r}\min\left\{\|u\|_{L^{q(\cdot)}(\Omega)},\|u\|\right\},\ \ \forall u\in X.
    \end{equation}
Using \eqref{4.2.Emb}, for $u\in X$ with $\|u\|>1$ we have
\begin{align}\label{4.2.Est.I/J.infty}
\notag \frac{I(u)}{J(u)}&\geq\frac{\frac{1}{p^+}\|u\|^{p^-}+\frac{1}{q^+}\min\left\{\|u\|_{L^{q(\cdot)}(\Omega)}^{q^+},\|u\|_{L^q(\cdot)(\Omega)}^{q^-}\right\}}{\frac{1}{r^-}\max\left\{\|u\|_{L^r(\cdot)(\Omega)}^{r^+},\|u\|_{L^r(\cdot)(\Omega)}^{r^-}\right\}}\\
&\geq \frac{\frac{1}{p^+}\|u\|^{p^-}+\frac{1}{q^+}\min\left\{\|u\|_{L^{q(\cdot)}(\Omega)}^{q^+},\|u\|_{L^{q(\cdot)}(\Omega)}^{q^-}\right\}}{\frac{C_{r}^{r^+}}{r^-}\max\left\{\|u\|_{L^{q(\cdot)}(\Omega)}^{r^+},\|u\|_{L^{q(\cdot)}(\Omega)}^{r^-}\right\}}.
\end{align}
Let $\{u_n\}\subset X\setminus\{0\}$ be any sequence such that $\|u_n\|\to\infty$ as $n\to\infty.$ If $\|u_n\|_{L^{q(\cdot)}(\Omega}\to\infty$ then, $\frac{I(u_n)}{J(u_n)}\to\infty$ due to \eqref{4.2.Est.I/J.infty} and the fact that $r^+<q^-$. If, up to a subsequence, $\{\|u_n\|_{L^{q(\cdot)}(\Omega}\}$ is bounded, then we also have $\frac{I(u_n)}{J(u_n)}\to\infty$ due to \eqref{4.2.Est.I/J.infty}. That is, \eqref{lim.infty} holds.

Next, we prove \eqref{lim.0}. Invoking \eqref{4.2.Emb} again, for $u\in X$ with $0<\|u\|<1$ we have
\begin{equation}\label{4.2.Est.I/J.0}
\displaystyle  \frac{I(u)}{J(u)}\geq\frac{\frac{1}{p^+}\|u\|^{p^+}}{\frac{1}{r^-}\max\left\{\|u\|_{L^{r(\cdot)}(\Omega}^{r^+},\|u\|_{L^{r(\cdot)}(\Omega}^{r^-}\right\}}\geq \frac{r^-\|u\|^{p^+}}{p^+C_{r}^{r^+}\|u\|^{r^-}}.
\end{equation}
Then, \eqref{lim.0} follows from \eqref{4.2.Est.I/J.0} and the fact that $p^+<r^-$.

\end{proof}
\begin{lemma}
    The infimum $\gamma_1$ is achieved at some $u\in X\setminus\{0\}.$ Moreover, $(u,\gamma_1)$ is an eigenpair of problem~\eqref{val}.
\end{lemma}
\begin{proof}
    Let $\{u_n\}\subset X\setminus\{0\}$ such that
    \begin{equation}\label{lim.gamma_1}
    \displaystyle \lim_{n\to\infty} \frac{I(u_n)}{J(u_n)}=\gamma_1>0.
    \end{equation}
From this and \eqref{lim.infty} it follows that $\{u_n\}$ is bounded in $X$. Thus, up to a subsequence we have $u_n\rightharpoonup u$ in $X$. Since $X\hookrightarrow\hookrightarrow L^{r(\cdot)}(\Omega)$, we easily deduce that
\begin{equation}\label{4.2.lim.J(u_n)}
\lim_{n\to\infty} J(u_n)=J(u).
\end{equation}
On the other hand, the continuity and the convexity of $I$ on $X$ imply that $I$ is weakly lower semicontinuous on $X$. Thus, we have
\begin{equation}\label{4.2.lim.I(u_n)}
\underset{n\to\infty}{\lim\inf}\,I(u_n)\geq I(u).
\end{equation}
We claim that $u\ne 0$. Indeed, suppose by contradiction that $u=0$. Then, \eqref{4.2.lim.J(u_n)} gives $\lim_{n\to\infty}J(u_n)=0.$ Combining this and \eqref{lim.gamma_1} we easily obtain that $\lim_{n\to\infty}I(u_n)=0$ and hence,
$\lim_{n\to\infty}\|u_n\|=0.$ From this and \eqref{lim.0} jointly with \eqref{lim.gamma_1}, we arrive at a contradiction. That is, we have shown that $u\ne 0.$ Thus, it follows from \eqref{4.2.lim.J(u_n)} and \eqref{4.2.lim.I(u_n)} that
\begin{equation*}
\displaystyle \lim_{n\to\infty} \frac{I(u_n)}{J(u_n)}\geq \frac{I(u)}{J(u)}.
\end{equation*}
Combining this with \eqref{lim.gamma_1} and the definition of $\gamma_1$ gives
\begin{equation}\label{u-gamma_1}
\gamma_1=\frac{I(u)}{J(u)}.
\end{equation}
It remains to show that $(u,\gamma_1)$ is an eigenpair of problem~\eqref{val}.  From the definition of $\gamma_1$ and \eqref{u-gamma_1} we deduce that for any $v\in X$,
\begin{equation*}
\frac{\diff }{\diff t}\frac{I(u+tv)}{J(u+tv)}\bigg|_{t=0}=0.
\end{equation*}
By a simple computation, the last equality and \eqref{u-gamma_1} yiled
\begin{align*}
\displaystyle \int_{\mathbb{R}^{N}\times
    \mathbb{R}^{N}}&\frac{|u(x)-u(y)|^{p(x,y)-2}(u(x)-u(y))(v(x)-v(y))}{|x-y|^{N+sp(x,y)}}\diff x\diff y
+ \displaystyle \int_{\Omega} |u|^{p(x)-2}uv\diff x\\&+
\displaystyle \int_{\Omega} |u|^{q(x)-2}uv\diff x=\gamma_1
\displaystyle \int_{\Omega} |u|^{r(x)-2}uv\diff x.
\end{align*}
That is, $(u,\gamma_1)$ is an eigenpair of problem~\eqref{val}. The proof is complete.
\end{proof}
\begin{proof}[Proof of Theorem \ref{mainnon} completed] Let $\lambda\in (\gamma_1,\infty).$ Recall that $\Phi_\lambda$ is of class $C^1(X,\mathbb{R})$ and any nontrivial critical point of $\Phi_\lambda$ is a nontrivial solution of problem~\eqref{val}, i.e., $\lambda$ is an eigenvalue of problem~\eqref{val}. By \eqref{lim.infty}, it is clear that $\Phi_\lambda$ is coercive. Moreover, $\Phi_\lambda$ is weakly lower semicontinuous, and hence $\Phi_\lambda$ has a global minimum achieved at some $w_\lambda\in X$. Since $\lambda>\gamma_1$, we find $v_\lambda\in X\setminus\{0\}$ such that $\frac{I(v_\lambda)}{J(v_\lambda)}<\lambda$, i.e., $\Phi_\lambda(v_\lambda)<0$. This yields
$\Phi_\lambda(w_\lambda)<0$ and hence, $w_\lambda\ne 0.$ Putting $u_\lambda=|w_\lambda|$ we deduce that $\Phi_\lambda(u_\lambda)\leq \Phi_\lambda(w_\lambda)$, and hence $u_\lambda$ is also a global minimum point for $\Phi_\lambda$. Thus, $u_\lambda$ is a critical point of $\Phi_\lambda.$ That is, we have shown that any $\lambda\in (\gamma_1,\infty)$ is an eigenvalue of problem~\eqref{val} and problem~\eqref{val} admits a nontrivial nonnegative solution.

Finally, let $\lambda\in (0,\gamma_0).$ Assuming by contradiction that there exists a $u_{\lambda}\in X\setminus \{0\}$
such that
$$\langle I'(u_{\lambda}),v\rangle=\lambda\langle J'(u_{\lambda}),v\rangle, \ \ \forall v\in X. $$
Taking $v=u_{\lambda}$ in the above equality we get
$$\langle I'(u_{\lambda}),u_{\lambda}\rangle=\lambda\langle J'(u_{\lambda}),u_{\lambda}\rangle,$$
i.e.,
$$I_0(u_{\lambda})=\lambda J_0(u_{\lambda}).$$
Thus,
$$\lambda=\frac{I_0(u_{\lambda})}{J_0(u_{\lambda})}\geq \gamma_0,$$
a contradiction. The proof is complete.
\end{proof}


\appendix
\section*{Appendix. The Lebesgue spaces with variable exponents}\label{Appendix}
In this Appendix, we recall some necessary properties of the Lebesgue spaces with variable exponents. We refer to \cite{fan,radrep} and the references
therein.

Let $\Omega$ be a bounded Lipschitz domain in $\mathbb{R}^{N}$. Consider the set
$$C_+(\overline\Omega)=\{p\in C(\overline\Omega,\mathbb{R}):\, p(x)>1\;{\rm
    for}\; {\rm all}\;x\in\overline\Omega\}.$$ 
For any $p\in
C_+(\overline\Omega)$, denote
$$p^+=\sup_{x\in\Omega}p(x)\qquad\mbox{and}\qquad p^-=
\inf_{x\in\Omega}p(x)$$ 
and define the {\it variable exponent Lebesgue space} $L^{p(\cdot)}(\Omega)$ as
$$L^{p(\cdot)}(\Omega)=\left\{u:\ u\ \mbox{is
    measurable real-valued function},\
\int_\Omega|u(x)|^{p(x)}\;\diff x<\infty\right\}.$$ This vector space is
a Banach space if it is endowed with the {\it Luxemburg norm}, which
is defined by
$$\|u\|_{L^p(\cdot)(\Omega)}=\inf\left\{\mu>0:\;\int_\Omega\left|
\frac{u(x)}{\mu}\right|^{p(x)}\;\diff x\leq 1\right\}.$$
We point out that if $p(x)\equiv p\in [1,\infty)$ then the optimal choice in the above expression is $\mu=\|u\|_{L^{p(\cdot)}(\Omega)}$.

Let $p\in
C_+(\overline\Omega)$ and let $L^{q(\cdot)}(\Omega)$ denote the conjugate space of
$L^{p(\cdot)}(\Omega)$, where $$1/p(x)+1/q(x)=1.$$ If $u\in
L^{p(\cdot)}(\Omega)$ and $v\in L^{q(\cdot)}(\Omega)$ then  the following
H\"older-type inequality holds:
\begin{equation*}\label{Hol}
\left|\int_\Omega uv\;\diff x\right|\leq\left(\frac{1}{p^-}+
\frac{1}{q^-}\right)\|u\|_{L^{p(\cdot)}(\Omega)}\|v\|_{L^{q(\cdot)}(\Omega)}\,.
\end{equation*}
Moreover, if $p_j\in C_+(\overline\Omega)$ ($j=1,2,\ldots, k$) and
$$\frac{1}{p_1(x)}+\frac{1}{p_2(x)}+\cdots +\frac{1}{p_k(x)}=1,$$
then for all $u_j\in L^{p_j(\cdot)}(\Omega)$ ($j=1,\ldots ,k$) we have
\begin{equation*}\label{Hol1}
\left|\int_\Omega u_1u_2\cdots u_k\;\diff x\right|\leq\left(\frac{1}{p_1^-}+
\frac{1}{p_2^-}+\cdots +\frac{1}{p_k^-}\right)\|u_1\|_{L^{p_1(\cdot)}(\Omega}\|u_2\|_{L^{p_2(\cdot)}(\Omega}\cdots \|u_k\|_{L^{p_k(\cdot)}(\Omega}\,.
\end{equation*}

An important role in manipulating the generalized Lebesgue spaces is
played by the {\it modular} of the $L^{p(\cdot)}(\Omega)$ space, which
is the mapping $\rho: L^{p(\cdot)}(\Omega) \rightarrow \mathbb{R} $
defined by
$$\rho(u)=\displaystyle \int_{\Omega}|u|^{p(x)}\diff x.$$

\noindent \textbf{Proposition A.1.}  \textit{It hold that:}
\begin{itemize}
	\item [(i)]$\|u\|_{L^{p(\cdot)}(\Omega)}<1 (=1;>1)\iff\rho(u)<1(=1;>1)$.
	\item [(ii)] $\|u\|_{L^{p(\cdot)}(\Omega)}>1 \iff \|u\|_{L^{p(\cdot)}(\Omega)}^{p^{-}}\leq \rho(u) \leq
	\|u\|_{L^{p(\cdot)}(\Omega}^{p^{+}}$.
	\item [(iii)] $\|u\|_{L^{p(\cdot)}(\Omega)}<1 \iff \|u\|_{L^{p(\cdot)}(\Omega)}^{p^{+}}\leq \rho(u)
	\leq \|u\|_{L^{p(\cdot)}(\Omega)}^{p^{-}}$.
\end{itemize}

\noindent \textbf{Proposition A.2.}  \textit{If $u,u_{n}\in L^{p(\cdot)}(\Omega)$ ($n\in \mathbb{N}$), then the
	following statements are equivalent to each other:}
\begin{itemize}
	\item [(1)] $\displaystyle \lim_{n\rightarrow \infty}
	\|u_{n}-u\|_{L^{p(\cdot)}(\Omega)}=0$.
	\item [(2)] $\displaystyle \lim_{n\rightarrow \infty} \rho(
	u_{n}-u)=0.$
\end{itemize}


\end{document}